\documentclass{amsart}
\usepackage{palatino, mathpazo}
\usepackage{mathrsfs}

\usepackage[colorlinks=true,]{hyperref}

\usepackage{amsthm, amsmath, amscd,amssymb}
\usepackage[all]{xy}

\newtheorem{thm}{Theorem}[section]
\newtheorem{defn}[thm]{Definition}
\newtheorem{prop}[thm]{Proposition}
\newtheorem{lem}[thm]{Lemma}

\theoremstyle{remark}
\newtheorem{remark}[thm]{Remark}
\newtheorem{rmk}[thm]{Remark}

\newcommand{\bA}{\mathbb{A}}
\newcommand{\bC}{\mathbb{C}}
\newcommand{\bP}{\mathbb{P}}
\newcommand{\bL}{\mathbb{L}}
\newcommand{\bT}{\mathbb{T}}

\newcommand{\bQ}{\mathbb{Q}}

\newcommand{\cA}{\mathcal{A}}

\newcommand{\cM}{\mathcal{M}}
\newcommand{\cI}{\mathcal{I}}
\newcommand{\cO}{\mathcal{O}}
\newcommand{\cP}{\mathcal{P}}

\newcommand{\cT}{\mathcal{T}}

\newcommand{\sF}{\mathscr{F}}

\newcommand{\fC}{\mathfrak{C}}
\newcommand{\fD}{\mathfrak{D}}
\newcommand{\fE}{\mathfrak{E}}
\newcommand{\fF}{\mathfrak{F}}

\newcommand{\fT}{\mathfrak{T}}
\newcommand{\fX}{\mathfrak{X}}
\newcommand{\fY}{\mathfrak{Y}}

\newcommand{\hilb}{\mathrm{Hilb}}

\newcommand{\D}{M^{\circ}}

\newcommand{\lp}{\left(}
\newcommand{\rp}{\right)}
\newcommand{\vir}{\mathrm{vir}}
\newcommand{\opk}{\mathrm{op}K^{0}}
\newcommand{\qcoh}{\mathrm{qcoh}}
\newcommand{\coh}{\mathrm{coh}}

\newcommand{\fix}{\mathrm{fix}}
\newcommand{\mov}{\mathrm{mov}}
\newcommand{\kth}{\iota_k}

\DeclareMathOperator{\rhom}{R\mathcal{H}om}

\DeclareMathOperator{\Spec}{\mathrm{Spec}}

\DeclareMathOperator{\gl}{gl}
\DeclareMathOperator{\pr}{pr}
\DeclareMathOperator{\ev}{ev}
\DeclareMathOperator{\id}{id}
\DeclareMathOperator{\can}{can}

\DeclareMathOperator{\tor}{Tor}
\DeclareMathOperator{\Coh}{Coh}

\begin{document}

\title{virtual pullbacks in $K$-theory}
\begin{abstract}
We consider virtual pullbacks in $K$-theory, and show that they are bivariant classes and satisfy certain functoriality.
As applications to $K$-theoretic counting invariants, we include proofs of a virtual localization formula for schemes and a degeneration formula in Donaldson-Thomas
theory.

\end{abstract}
\author{F.~Qu}
\address{}
\email{fengquest@gmail.com}


\maketitle

\section*{Introduction}

\subsection{}

Virtual pullbacks were introduced and developed for Chow groups in \cite{vpb}, we
work out  parallel results for $K_0$ groups of coherent sheaves. $K$-theoretic virtual pullbacks also give rise to bivariant classes (cf. \cite[Definition 4.1]{AS})
and satisfy functoriality.  To prove these results, we follow the arguments in \cite{vpb}, \cite{KKP}, and \cite{Fu}.

As localization and degeneration techniques are fundamental in curve-counting theories, we also include  
proofs of a $K$-theoretic virtual localization formula for schemes and a degeneration formula in Donaldson-Thomas (DT) theory. These formulas are known and straightforward to prove given their cycle versions. 
For applications to $K$-theoretic computations, see e.g., \cite{Ok}.

\subsection{}
We work over a field $k$, 
 schemes and algebraic stacks are over $k$ and (locally) of finite type.


\subsection{}
The mechanism of virtual pullbacks is the same as that of Gysin pullbacks along regular embeddings.

Given a closed immersion between schemes $f\colon X \to Y $, we have a deformation space
$\D_f$.  It is a scheme flat  over $\bP^1$ and the diagram below is cartesian
\begin{equation} \label{ds}
\xymatrix{
C_f \ar@{^{(}->}[r]^{i} \ar[d] & \D_f \ar[d] & Y\times \bA^1\ar[d]\ar@{_{(}->}[l]_{j}\\
\{\infty\} \ar@{^{(}->}[r]  & \bP^1 & \bA^1\ar@{_{(}->}[l] &.
}
\end{equation}
Here $C_f$ is the normal cone of $f$.
 (See \cite[Chapter 5]{Fu}.)
 
When $Y$ is of finite type,  
we can define deformation to the normal cone map
\[
\sigma_f= i^*\circ {j^*}^{-1} \circ \pr^*: A(Y) \to A(C_f)
\] using the diagram
\begin{equation} \label{loc}
\xymatrix{
A(C_f) \ar[r]^{i_*} & A(\D_f) \ar[r]^{j^*} \ar[d]^{i^*} & A(Y \times \bA^1) \ar[r] & 0 \\
                 &A(C_f)           & A(Y)\ar[u]_{\pr^*}\ar@{-->}[l]_{\sigma_f}             &. \\
} 
\end{equation}
Here $A(\cdot)$ denotes the Chow group functor.

For any cartesian diagram
\begin{equation}
\xymatrix{
X' \ar[r]^-{g}\ar[d] & Y'\ar[d]\\
X \ar[r]^f & Y,
}
\end{equation}
we have a closed immersion $\iota\colon C_g \hookrightarrow C_f \times_X X'$.
When $f$ is a regular embedding, $C_f$ is a vector bundle, and the map $\iota$ embeds
$C_g$ into a vector bundle over $X'$. 
Now we can define the Gysin pullback
\[
\xymatrix{ g^!\colon A(Y')\ar[r]^-{\sigma_{g}}& A(C_{g} )\ar[r]^-{\iota_*} & A( C_f \times_X X') \ar[r]& A(X').
}\]
The last map $A( C_f \times_X X') \to A(X')$ is the Thom isomorphism.

Then the pullbacks $\{ g^!\colon A(Y') \to A(X')\}$ defines a bivariant class and such 
classes further satisfy a functoriality. Recall a bivariant class (\cite[Chapter 17]{Fu}) for $f\colon X \to Y$ is given by a collection of maps 
$\{ c_\nu \colon A(Y') \to A(X \times_Y Y') \}$ indexed by $\nu \colon Y' \to Y$, compatible with
proper pushforwards, flat pull backs, and the functoriality is the statement that
for a composition of regular embeddings
\[
\xymatrix{
X \ar[r]^i & Y \ar[r]^j  &Z, 
}
\]
we have $i^! \circ j^! = (j\circ i)^!$. 

It is clear that to define the bivariant class $f^!$, the ingredients are deformation spaces, embeddings of normal cones into vector bundles, and a homology theory.
As perfect obstruction theories induce embeddings of normal cones into vector bundle 
stacks, and deformation spaces and Chow groups are extended to Artin stacks by Kresch's work \cite{K1,K2}, the above construction can be generalized.

More precisely, given a map $f\colon X \to Y $ between algebraic stacks of finite type over $k$ such that $X \to X \times_Y X$ being unramified, we have an algebraic stack
$\D_f$ as in \eqref{ds}, with $C_f$ being the intrinsic normal cone (\cite{BF}) for $f$.

We have deformation to the normal cone map
\[\sigma_f\colon A(Y) \to A(C_f).\]
Together with a closed embedding $\iota\colon C_f \to \fE_f$ of $C_f$ into a vector stack $\fE_f$, 
the virtual pullback
\[
\xymatrix{ f^!\colon A(Y)\ar[r]^{\sigma_{f}}& A(C_f) \ar[r]^-{\iota_*} & A( \fE_f) \ar[r]& A(X')
}\] is introduced in \cite{vpb}.

\begin{remark}
A closed embedding $\iota\colon C_f \to \fE_f$
corresponds to a perfect obstruction theory for $f$.

The functoriality of virtual pullbacks depends on compatibilities between perfect obstruction theories (\cite{BF,vpb}).
See Proposition \ref{prop:func} below for a precise statement.

\end{remark}


\subsection{}
In this note, instead of Chow groups, we work with $K_0$ groups of coherent sheaves.
In Section 1, we recall relevant definitions including DM morphisms, perfect obstruction theories, and bivariant classes, and collect some results on 
$K_0$ groups of algebraic stacks and deformation spaces.
Section \ref{s2} concerns virtual pullbacks. Bivariance follows from properties of the deformation space functor $\D$, while functoriality relies  furthermore on
\cite[Proposition 1]{KKP} and requires some efforts to prove. In Section \ref{s3},  
a localization formula for schemes is proved by the method of \cite{CKL}. In  Section \ref{s4}, we indicate
how arguments in \cite{LW,MPT} lead to a degeneration formula in DT theory.

\section{Preliminaries}

\subsection{DM morphisms}
A morphism $f\colon X \to Y$ between algebraic stacks (
\cite[Tag 026O]{Stacks}) 
is DM 
(\cite[Tag 04YW]{Stacks})
if $\Delta_f\colon X \to X \times_Y X$ is unramified. 
Then for any morphism $Z \to Y$ from an algebraic space $Z$, $X \times_Y Z$ is a DM stack 
(\cite[
03YO]
{Stacks}).
In particular, when $X$ is a DM stack, $f$ is DM.

When $f$ is DM, we can represent it as a map between groupoids (in algebraic spaces)
$f_\bullet\colon X_\bullet \to Y_\bullet$ such that $f_0\colon X_0 \to Y_0$ and $f_1\colon X_1 \to Y_1$ are unramified. In fact, there exists a commutative diagram
\[
\xymatrix{
X_0 \ar[r]^{f_0}\ar[d]  & Y_0\ar[d]\\
X \ar[r]               & Y
}\] such that vertical arrows are smooth surjective and $f_0$ is a disjoint union of closed immersions between affine schemes
(cf. \cite[Lemma 2.27]{vpb}). Then $f_1=f_0 \times_f f_0$ is unramified, which is easy to see using
the diagram
\[
\xymatrix{
X_1=X_0\times_X X_0 \ar[r]\ar[d] & X_0 \times_{Y} X_0\ar[d]\ar[r] & Y_1=Y_0\times_Y Y_0 \\
X \ar[r]             & X \times_Y X.
}
\]



\subsection{Deformation spaces}

To each DM morphism $f\colon X \to Y$ between algebraic stacks, we have a deformation space $\D_f$.
It is a flat family over $\bP^1$ whose fiber over $\{\infty\}$ is the intrinsic normal $C_f$, and over 
$\bA^1=\bP^1-\{\infty\}$,
 it is isomorphic to the product $Y\times \bA^1$.
 
For a closed immersion between schemes, $\D_f$ is constructed in \cite[Chapter 5]{Fu}.
In general, $\D_f$ is constructed by descent  (\cite{K1,K2,KKP}).
First, the construction of $\D$ as algebraic spaces for unramified morphisms between algebraic spaces is achieved by
using \'etale groupoids in schemes, as unramified morphisms are \'etale locally immersions. In general, we can represent $f$ as a map between groupoids
$f_\bullet\colon X_\bullet \to Y_\bullet$, such that $f_0,f_1$ are unramified, then
$\D_f$ is the stack associated to the smooth groupoid $\D_{f_1} \rightrightarrows\D_{f_0}$. 

\begin{lem} \label{qcqs}
Given a DM morphism $f\colon X \to Y$.  The deformation space
 $\D_f$ is quasi-compact, quasi-separated (qcqs) if and only if
$X$ and $Y$ both are.
\end{lem}

\begin{proof}
We will apply results in \cite[Tag 075S]{Stacks} implicitly numerous times in this proof.

For the if direction, as $X$ is quasi-compact, $Y$ is qcqs, and we can represent $f$ as a morphism between groupoids $f_\bullet\colon X_\bullet \to Y_\bullet$
such that $X_0 \to Y_0$ is a closed immersion between affine schemes.
Then $\D_{f_0}$ is qcqs, and $\D_f$ is quasi-compact.
 
 As $X$ and $Y$ are qcqs, $X_0$ and $Y_0$ are affine, we see that $X_1$ and  $Y_1$ 
 are qcqs algebraic spaces. 

Assume $\D_{f_1}$ is qcqs for the moment,  as $\D_{f_0}$ is  qcqs, $\D_{f_1} \to \D_{f_0} \times \D_{f_0}$ is qcqs,
it follows that $\D_f$ is quasi-separated. 
To show that $\D_{f_1}$ is qcqs, represent it as a \'etale groupoid of immersions, and run the argument above again.
 
For the only if direction, if $\D_f$ is qcqs, then $\D_{f} \to \bP^1$ is qcqs, so its fibers $C_f$ and $Y$ are qcqs.
As the zero section of a cone stack is qcqs,  
$X \to \Spec k$ as the composition of $X \to C_f$ and $C_f \to \Spec k$ is qcqs.

 \end{proof}

\begin{prop}[{cf. \cite[Theorem 2.31]{vpb}}] \label{defcon} 
Given a cartesian diagram of algebraic stacks
\[
\xymatrix{
X' \ar[r]^g\ar[d]  & Y' \ar[d]^-\nu \\
X \ar[r]^f              & Y
} 
\] such that $f$ is DM, there is an induced map $\D_g \to \D_f$
over $Y' \to Y$.
$\D_g \to \D_f \times_{Y}Y'$ is an isomorphism when $\nu$ is flat, and
a closed immersion in general.
\end{prop}
\begin{remark}
The proposition is not hard to show assuming the construction of $\D_f$ is independent of
groupoid presentations of $f$.

Presumably, the deformation space construction gives rise to a functor $\D$  from the $(2,1)$-category of morphisms between algebraic stacks to the $(3,1)$-category of algebraic 2-stacks, and $\D_f$ is a 1-stack when $f$ is DM.  Assuming the expected properties of $\D$, one can introduce virtual pullbacks for Artin stacks involving 2-stacks. 
If we truncate 2-stacks to 1-stacks (i.e., taking $\pi_{\le 1}$.), then 
we have the version of virtual pullbacks in \cite{Po}, which is similar to working with obstruction sheaves instead of vector bundle stacks. 
We hope to address these matters in \cite{Q}.
\end{remark}

\subsection{Grothendieck groups of coherent sheaves}
\subsubsection{}
We will use $K_0(-)$ to denote the Grothendieck group of an abelian category or a triangulated category.
Recall the Grothendieck group of an abelian category $\cA$ is 
the abelian group generated by symbols $[a]$ for each object $a$ in $\cA$ modulo relations generated by 
\[
[a]=[a']+[a'']
\] for each exact sequence 
\[0\to a'\to a\to a''\to 0.\]
The Grothendieck group of a triangulated cateogory $D$ is defined similarly, it is 
the abelian group  generatored by $[x]$ for objects $x$ in $D$ and relations
\[
[x]=[x']+[x'']
\] for each distinguished triangle $x'\to x\to x''$.

Let $D$ be a triangulated category with a $t$-structure. Denote $\cA_D$ its heart and 
$D^b$ the full subcategory of $D$ consisting of bounded objects, i.e., $x \in D$ such that $H^n(x)=0$ for $|n|>>0$, here $H^n=\tau_{\le n}\tau_{>n}\colon D \to \cA_D$.
Note that there is an isomorphism \[K_0(D^b) \simeq K_0(\cA_D)\]
given by 
\[
[x] \mapsto \sum (-1)^i[H^i(x)].\] 

If we have a triangulated functor $\sF\colon D \to E$ such that $\sF(D^b)\subset \sF(E^b)$, then we have an induced functor 
$K_0(D^b) \to K_0(E^b)$, or equivalently  a functor $K_0(\cA_D) \to K_0(\cA_E)$.


\subsubsection{}

For an algebraic stack $X$ locally of finite type over $k$,
denote by $K_0(X)=K_0(\Coh(X))$, where $\Coh(X)$ is the  abelian category of coherent sheaves on $X$.

\begin{rmk}
(Quasi) Coherent sheaves can be defined using the lisse-\'etale site of $X$ as in \cite[Definition 6.1]{O1}. 
See, e.g., \cite[Section 1]{HR}, for a summary of 
quasi-coherent sheaves on algebraic stacks. 

As maps between $K_0$
groups are induced by derived functors, it is more flexible to 
to think of $K_0(X)$ as $K_0(D^b_\coh(X))$, here $D^b_\coh(X)$ is the full subcategory
of the derived category of $\cO_X$-modules with coherent cohomology. 
\end{rmk}

For  a flat morphism $f\colon X \to Y$, we have the pullback
$f^*\colon \Coh(Y) \to \Coh(X)$, since its exact by the flatness assumption, we have an induced map
$f^!\colon K_0(Y) \to K_0(X)$.


For a proper map $f\colon X \to Y$, we have $R^if_* F\in \Coh(Y)$ for any coherent sheaf $F$ on $X$ and each $i \ge 0$ by \cite[Theorem 1.2]{O2}, \cite[Theorem 1]{Fa}.
Therefore the map \[
Rf_*\colon D^+_\qcoh(X) \to D^+_\qcoh(Y)
\] induces
\[
Rf_*\colon D^b_\coh(X) \to D^+_\coh(Y).
\]
If $Rf_*$ satisfies
\begin{equation} \tag{$\dagger$} \label{proper}
Rf_* (D^b_\coh(X)) \subset  D^b_\coh(Y), 
\end{equation}

then we can define $f_*\colon K_0(X) \to K_0(Y)$ by 
\[
[F] \mapsto \sum_n (-1)^n [R^nf_*F].
\] 
 
Because of the condition on the pushfoward map above, we decided to 
consider pushforwards only along proper DM maps for simplicity.

\begin{rmk}
It is easy to see \eqref{proper} is the same as the condition 
\[
Rf_*(\Coh(X)) \subset  D^b_\coh(Y).
\]

Two related notions are 'of finite cohomological dimension' (\cite[Definition 2.3]{HR}),
which requires $Rf_*(\operatorname{Qcoh}(X)) \subset D^{\le n}_\qcoh(Y)$ for some $n$, and concentrated(\cite[Definitino 2.4]{HR}), which is
similar to being universally of finite cohomological dimension.
Obviously, a proper map of finite cohomological dimension satisfies \eqref{proper}. 

If $f$ is proper DM, then it is concentrated, in particular, satisfies \eqref{proper},
this follows from \cite[Theorem 2.1]{HR0}, or 
one can employ coarse moduli spaces.
\end{rmk}

The functor $K_0(-)$ is covariant with respect to proper DM morphisms, contravariant with respect to flat morphisms. Proper pushfowards commute with flat pullbacks by, e.g., \cite[Lemma 1.2 (4)]{HR}.

\begin{rmk}
Covariance and contravariance are interpreted with respect to the homotopy category of stacks, as it is easy to see
that the flat pullback $f^!$ or the proper pushforward $f_*$ only depends on the homotopy class of $f$.

\end{rmk}

Let $X$ be an algebraic stack, quasi compact and quasi-separated 
\footnote{The map $X \to X \times X$ is quasi-compact and quasi-separated.}, 
locally of finite type over $k$, and $Z$ an closed substack of $X$ with complement $U$, then
we have the localization sequence
\[
K_0(Z) \to K_0(X) \to K_0(U) \to 0.
\]
This can be proved as if $X$ is a Noetherian scheme using \cite[Proposition 15.4]{LMB}.

For a morphism $i \colon X \to Y$ that is smooth locally a regular closed immersion between schemes, we have a Gysin pullback $i^!$.
Given a cartesian diagram
\begin{equation} \label{sqr}
\xymatrix{
X' \ar[r]\ar[d] 
   & Y' \ar[d] \\
X \ar[r]
   & Y
   }   
\end{equation}
$i^!\colon K_0(Y') \to K_0(X')$ is given by 
\[
i^![G] = \sum_n (-1)^n\tor_n^Y(G, \cO_X),
\] where $\tor_n^Y(G, \cO_X)$ is the tor sheaf.
Note that $i^!\colon K_0(Y) \to K_0(X)$ is given by 
$Li^*\colon D^b_\coh(Y) \to D^b_\coh(X)$.

In particular, for the zero section of a vector bundle stack,
 we have a Gysin pullback.
\begin{rmk}
For a summary of Tor sheaves, see, e.g., \cite[3.1, 3.2]{AS}.
To extend results proved for tor sheaves on schemes to algebraic stacks, we note that the formation of $Tor$ in \eqref{sqr} behaves well under flat maps in $X$, $Y$, and $Y'$.
\end{rmk}


Gysin pullbacks commute with proper pushforwards and flat pullbacks (\cite[Lemma 3.1, Lemma 3.3]{AS}).

When $i\colon X \to Y$ is represented by a regular closed immersion, we have
\[
i^!i_*[F]= [F]\otimes \Lambda_{-1}(N_i^\vee)\colon K_0(X) \to K_0(Y),
\] where $N_i$ is the normal bundle of $i$.

For $F$ a coherent sheaf on $X$ and $G$ a coherent sheaf on $Y$, denote
by $F\boxtimes G$ the sheaf $\pr_X^*F \otimes \pr_Y^* G$ on $X \times Y$.
As the projection maps $\pr_X, \pr_Y$ are flat, we have an induced map
\[
\boxtimes\colon K_0(X) \times K_0(Y) \to K_0(X \times Y).
\]
\subsection{Bivariant classes}
The notion of an operational bivariant class for a representable map between quotient stacks is introduced in \cite{AS}. 
It is straightforward to adapt the definition there to algebraic stacks in general.

Let $f\colon X \to Y$ be a map between algebraic stacks, we have a group 
$\opk(X\xrightarrow{f} Y)$ of bivariant classes.
A bivariant class $c$ in $\opk(X\xrightarrow{f} Y)$ is given by a collection of maps
\[
c_\nu\colon K_0(Y') \to K_0(X \times_Y Y')
\] indexed by $\nu\colon Y' \to Y$.
These maps should commute with proper DM pushforwards, flat pullbacks, and Gysin pullbacks.

\begin{rmk}
Proper DM pushforwards are not too restrictive, considering pushforwards in Chow groups (with rational coefficients) are defined only
for proper DM morphisms.
\end{rmk}

\subsection{Perfect obstruction theories}
Given a morphism $f\colon X \to Y$ between algebraic stacks,  denote by $\bL_f \in D^{\le 1}_{\qcoh}(X)$ the cotangent complex of $f$. Here $D_{\qcoh}(X)$ is the full subcategory of the derived category of $\cO_X$-modules (on the lisse-\'etale site of $X$)
with quasi-coherent cohomology sheaves. Cotangent complexes for algebraic stacks
behave the same as those for schemes (\cite[2.4]{AOV}).

An obstruction theory for $f$ is given by a map $\phi\colon E^\bullet \to \bL_f$ in $D_\qcoh(X)$ such that
$h^1(\phi), h^0(\phi)$ are isomorphisms, $h^{-1}(\phi)$ is surjective. 
If $E^\bullet$ is a perfect complex of tor amplitude [-1,1], then it is called a perfect obstruction theory (POT) (\cite[Definition 3.1]{Po}).

When $f$ is DM, $\bL_f\in D^{\le 0}_{\qcoh}(X)$, and a perfect obstruction theory
$\phi\colon E^\bullet \to \bL_f$ induced a
closed embedding $C_f \hookrightarrow \fE_f$ between cone stacks, where 
$\fE_f=h^1/h^0({E^{\bullet}}^{\vee})$ (\cite{BF, Po}), and any such imbedding corresponds to some POT. So
a POT can be viewed either as some map in the derived category or 
an embedding of the intrinsic normal cone
into some vector bundle stack. We will switch between these two viewpoints freely.

For  a cartesian diagram
\[
\xymatrix{
X' \ar[r]^g\ar[d]^\mu & Y'\ar[d]\\
X \ar[r]^f & Y,
}
\]
a POT $E^{\bullet} \to \bL_f$ induces a POT $\mu^* E^{\bullet} \to \bL_g$ for $g$, it is given 
by the composition \[\mu^*E^{\bullet} \to \mu^* \bL_f \to \bL_g .\] 
The induced embedding of $C_g$ is given by the compostion
\[
C_g \hookrightarrow C_f\times_X X' \hookrightarrow \fE_f\times_X X'.
\]

\section{Virtual Pullbacks}
\label{s2}

In this section, all stacks are of finite type and quasi-separated over $k$.

\subsection{Deformation to the normal cone map}

Let  $f\colon X \to Y $ be a DM morphism between algebraic stacks.
As $\D_f$ is qcqs by Lemma \ref{qcqs}, we have a localization sequence to define the deformation to the normal cone map 
\[\sigma_f = i^*\circ {j^*}^{-1}\circ \pr^*\colon K_0(Y) \to K_0(C_f)\]
using $K$-theoretic version of \eqref{loc}. (See \cite[page 352]{Fu}.)

\begin{prop} \label{biv}
Consider a cartesian diagram between algebraic stacks
\[
\xymatrix{
X' \ar[r]^g\ar[d]  & Y' \ar[d]^\nu \\
X \ar[r]^f              & Y,
} 
\] where $f$ is DM. Let $\xi\colon C_{g} \to C_f$ be the induced map between cone stacks.
 
\begin{enumerate}
 \item if $\nu$ is proper DM, then
 \[ \xi_*\circ \sigma_g =\sigma_f\circ \nu_* \colon K_0(Y') \to K_0(C_f).\]
 
 \item if $\nu$ is flat, then $\xi$ is flat, and 
  \[ 
   \sigma_g\circ \nu^! = \xi^!\circ\sigma_f\colon K_0(Y) \to K_0(C_g) .\]
\end{enumerate}
\end{prop}

\begin{proof}

We treat the proper pushforward case, the flat pullback case is similar and easier.

As $Y' \to Y$ is proper DM, so is $\D_g \to \D_f$, since it is the composition of a closed immersion
$\D_g \to \D_f\times_{Y'}Y$ by Proposition \ref{defcon}, and proper DM map $\D_f\times_{Y'}Y \to \D_f$.
The map $\D_g \to \D_f$ induces  a commutative diagram
\[
\xymatrix{
K_0(C_g) \ar[r]\ar[d]^{\xi_*} & K_0(\D_g) \ar[r]\ar[d] & K_0(Y'\times\bA^1) \ar[r]\ar[d]&0\\
K_0 (C_f) \ar[r]         &K_0(\D_f)\ar[r]           & K_0(Y\times \bA^1)\ar[r]&0,
}
\] where horizontal arrows are localization sequences, and vertical arrows are proper pushforwards.
Using the commutativity between flat pullbacks, proper pushforwards, and Gysin pullbacks, a diagram chase gives the result of (1). 
\end{proof}


\subsection{Virtual pullbacks}

\begin{defn}[{cf. \cite[Definition 3.7]{vpb}} ]
Assume $f$ is DM,
a perfect obstruction theory (POT) $\phi\colon E^{\bullet} \to \bL_f^{\bullet} $\,  for $f\colon X \to Y$ gives rise to  a closed immersion
$\iota: C_f \hookrightarrow \fE_f$, where $\fE_f = h^1/h^0({E^{\bullet}}^{\vee}).$
Using this immersion, we can define a virtual pullback 
\[f^{!}\colon K_0(Y) \to K_0(X)\]
 as the composition:
\[
\xymatrix{ K_0(Y) \ar[r]^-{\sigma_f} &K_0(C_f) \ar[r]^-{\iota_*} &K_0(\fE_f)\ar[r]^{s^!}& K_0(X),}
\]
 where
$s$ is the zero section of $\fE_f$, $s^!$ its Gysin pullback.

The virtual structure sheaf $\cO_\phi$ is defined as  $f^!(\cO_Y)$.

\end{defn}

\begin{rmk}
When $f$ is smooth and DM, the identity map of $\bL_f$ gives rise to a POT, and the virtual pullback is the same as the flat pullback.
\end{rmk}

\subsection{Bivariance}

Consider a cartesian diagram
\[
\xymatrix{
X' \ar[r]^{g}\ar[d]^\mu & Y'\ar[d]\\
X \ar[r]^f & Y,
}
\]
a POT $E^{\bullet} \to \bL_f$ induces a POT $\mu^* E^{\bullet} \to \bL_{g}$ for $g$, 
so we have a map 
\[g^!: K(Y') \to K(X').\]

We will show that the collection of maps 
$g^!: K(Y') \to K(X')$ for each $Y' \to Y$ defines a bivariant class, denoted by $f^!\in  \opk(X\xrightarrow{f} Y)$.

\begin{prop}[cf. {\cite[Example 17.6.4]{Fu}}]
$f^!$ commutes with proper DM pushforwards  and flat pullbacks.
\end{prop}

\begin{proof}
This follows from Proposition \ref{biv}.

\end{proof}


\begin{prop}[Commutativity]
Given a cartesian diagram
\[\xymatrix{
X' \ar[r]\ar[d] & Y'\ar[d]^{\nu}\\
X \ar[r]^{f} & Y
}\] such that we have defined virtual pullback $f^!,\nu^!$, 
then $f^! \circ \nu^! = \nu^! \circ f^!$. 
\end{prop}

\begin{proof}

It is enough to show 
\[
f^! \circ \nu^! = \nu^! \circ f^!\colon K_0(Y) \to K_0(X').
\]
Consider the cartesian diagram
\begin{equation} \label{eq1}
 \xymatrix{
C_f \times_Y C_\nu \ar[r]\ar[d] & g^*C_\nu \ar[r]^{h}\ar[d] & C_\nu\ar[d]\\
\mu^*C_f \ar[d]^{\xi}\ar[r]&X' \ar[r]^{g}\ar[d]^{\mu} & Y'\ar[d]^{\nu}\\
C_f\ar[r]&X \ar[r]^{f} & Y,
}
\end{equation}
Unravel the definition, it is easy to see that
$\nu^!\circ f^! $ is the composition of $\sigma_f$, $\sigma_\xi$
pushfoward along $C_\xi \hookrightarrow C_f \times_Y C_\nu \hookrightarrow \fE_f \times_Y\fE_\nu$, and Gysin pullback along $X' \to \fE_f \times_Y\fE_\nu$.

Consider the double deformation space
$\pi\colon M_f \times_Y M_\nu \to \bP^1 \times \bP^1$, and principal cartier divisor $D, E$ on 
$M_f \times_Y M_\nu$ correspondes to $\{\infty\} \times \bP^1$ and $\bP^1 \times \{\infty\}$.
Given a coherent sheaf $F$ on $Y$, pullback it to $\pi^{-1}(\bA^2)=Y \times \bA^2$, then extends to $\tilde{F}$ on $M_f \times_Y M_\nu$, one can check that 
the pushforward of $\sigma_\xi \circ \sigma_f([F])$ along
$C_\xi \hookrightarrow C_f\times_Y C_\nu$ is given by  $i_E^! \circ i_D^! ([\tilde{F})]$ in $K(C_f \times_Y C_\nu)$.

Now we see that proposition is a consequence of 
$ i_D^! \circ i_E^! =i_E^!\circ i_D^!.$

\end{proof}

\begin{thm}
Virtual pullbacks are bivariant classes. 
\end{thm}

\begin{proof}
Since we have proved virtual pullbacks commute with proper DM pushforwards, flat pullbacks, and
Gysin pullbacks, they are bivariant classes.
\end{proof}

\begin{rmk}
In fact, as a virtual pullback
is determined by flat pullbacks, proper pushforwards and Gysin pullbacks,
one can show that virtual pullbacks commute with bivariant classes. In particular, they
commute with refined Gysin maps (See e.g., \cite[Section 3]{AS})
\end{rmk}

\subsection{Functoriality}

We will need the local description of deformation spaces.
Consider $f\colon \Spec(A/I) \to \Spec A$,
recall $\D_f$ over $\bP^1-\{0\}$ is given by the $k[T]$ algebra 
\[
A[T]\oplus\bigoplus_{n>0} \frac{I^n}{T^n} \subset A[T,T^{-1}],
\]
and we see that over $P^1-\{0\}-\{\infty\}$ where $\{\infty\}$ is the point $T=0$, we get
the $k[T, T^{-1}]$ algebra $A[T,T^{-1}]$.

\begin{lem}
Let $C$ be a cone stack over an algebraic stack $Y$, and $s\colon Y \to C$ the zero section, 
then the deformation space  $\D_{s}$ is given by the associated $C$ bundle over $\bP^1$ of the principal
$G_m$ bundle over $\bP^1$ determined by $\cO(-1)$. Here the $G_m$ action on $C$ is induced by the multiplicative action of $\bA^1$ on $C$ as a cone stack.
In particular, the intrinsic normal cone $C_s$ is isomorphism to $C$, and $\sigma_s\colon K_0(C) \to K_0(C_s)$ is the identity.

\end{lem}

\begin{proof}
First consider the case when $Y$ is a scheme and $C$ is a cone. 
Locally $Y$ is given by
an affine scheme $\Spec A$ and $C =\Spec S_\bullet$, where $S_\bullet$ is an $A$ algebra generated over $S_0=A$ by $S_1$. 
As $Y \to C$ is given by the ideal $S_+=\oplus_{n>0}S_n$, $\D_{s}$ over $\bP^1-\{0\}$
is given by 
\[
S_\bullet[T] \oplus\bigoplus_{n>0}\frac{S_{+}^n}{T^n} = (\oplus_{d\ge 0}\frac{S_d}{T^d})[T].
\] 
 There is an isomorphism
\begin{equation} \label{coniso}
 S_\bullet[T] \simeq  (\oplus_{d\ge 0}\frac{S_d}{T^d})[T]
\end{equation} 
functorial in $S_\bullet$,
which maps any element $x$ in $S_d$ to $\frac{x}{T^d}$, and $T$ to $T$.

Consider the isomorphism \eqref{coniso} over $\bP^1-\{0\}-\{\infty\}$, the right hand side
is isomorphic to $S_\bullet[T,T^{-1}]$ via 
\[
S_\bullet[T] \oplus\bigoplus_{n>0}\frac{S_{+}^n}{T^n}  \subset S_\bullet[T, T^{-1}]
\]
and \eqref{coniso} induces
\[
S_\bullet[T,T^{-1}] \simeq S_\bullet[T,T^{-1}]
\]
that corresponds to
the isomorphism
\[C \times \bA^1-\{0\} \to C \times \bA^1-\{0\}\] that maps
$(a, \lambda)$ to $(\lambda^{-1}a, \lambda)$.
Now we see $\D_{s}$ is the pushout
\[
\xymatrix{
C\times \bA^1-\{0\}  
  \ar[rr]^-{      (a,t^{-1})\mapsto (a, t^{-1})    }     \ar[d]^{    (a,t^{-1}) \mapsto (t^{-1}a, t)   }
 & 
 &C \times \bP^1-\{\infty\}\\
C\times \bP^1-\{0\}&.
}
\] 
Here $t$ is the coordinate on $\bP^1-\{0\}$.

The map $\D_{s} \to C$ is given by projection to $C$ over $\bP^1 -\{\infty\}$, and
 $(b, t) \mapsto bt$ over $ \bP^1-\{0\}$.
   
By the functorial nature of these identifications and the descent construction of deformation spaces, we see that the lemma works for $Y$ an algebraic space and $C$ a cone over $Y$. In general, first choose a smooth cover of $Y$ by a scheme $U \to Y$, such that
 $C \times_Y U$ has a global presentation $[D/E]$ as a cone stack, then $s$ is represented by the groupoid
 \[
 \xymatrix{
 U\times_YU \ar[r]\ar@<.5ex>[d]\ar@<-.5ex>[d]  & D\times_CD\ar@<.5ex>[d]\ar@<-.5ex>[d]\\
 U \ar[r]                          & D,
 }
 \] and we are back to the case for cones over algebraic spaces.
 
To see that $\sigma_s$ is the identity, note that over $\bP^1 -\{0\}$,
$\D_f$ is given by $C \times \bA^1$, and the pullback to $C$ via $\{t\} \to \bA^1$ is independent of $t$.

 \end{proof}
 
 \begin{rmk}
 The proof also shows that given a commutative diagram
 \[
 \xymatrix{
 X \ar[r]^s\ar[d]^f & C\ar[d]^\xi\\
 X' \ar[r]^{s'} & C'
 }
 \] where the horizontal arrows are zero sections of cone stacks, and $\xi$ is equivariant with respect to their $\bA^1$ action, the induced map $\D_s \to \D_{s'}$ over $\bP^1$ is given fiberwise by $ \xi\colon C \to C'$.
 \end{rmk}

\begin{lem} \label{split}
Given a map $f\colon X \to Y$, a stack $ \pi\colon C \to Y$  over $Y$ with a section $s: Y \to C$.

\begin{enumerate}
\item 
The triangle
\[\xymatrix{f^*\bL_s          \ar[r]            &  \bL_{s\circ f} \ar[r] &        \bL_f}\]
associated to $X \to Y \to C$ 
is isomorphic to 
\[\xymatrix{f^*\bL_s          \ar[r]            &  f^*\bL_s \oplus \bL_f \ar[r] &        \bL_f},\]
\item We have an induced closed immersion
\[
C_{s\circ f} \to C_f \times_X f^*C_s.
\]
\item 
Assume $E_f\to \bL_f$ (resp. $E_s \to \bL_s$)  is a POT for $f$ (resp. $s$).
Then we can construct a compatible triple

\[ 
\xymatrix{
f^*E_s^{\bullet} \ar[r]\ar[d] & f^*E_s^{\bullet}\oplus E_f^{\bullet}\ar[d]\ar[r] &E_f^{\bullet} \ar[d]\\
f^*\bL_s          \ar[r]            & f^*\bL_s \oplus \bL_f \simeq \bL_{s\circ f} \ar[r]  &        \bL_f,
}
 \]

Let $\mathfrak{E}_f$ ($\mathfrak{E}_s$ ) be $h^1/h^0({E^{\bullet}_f}^{\vee})$ ($h^1/h^0({E^{\bullet}_s}^{\vee})$),
then the induced closed immersion 
\[
C_{s\circ f} \hookrightarrow \mathfrak{E}_f\times_X f^*\mathfrak{E}_s
\] 
from the middle vertical arrow
is given by the composition
\[
C_{s \circ f}  \hookrightarrow  C_f \times_X f^*C_s \hookrightarrow
\mathfrak{E}_f \times_X f^* \mathfrak{E}_s.\]

\end{enumerate}
\end{lem}

\begin{proof}
(1)The diagram
\[
\xymatrix{
X\ar[r] \ar[d]^f & X\ar[r]\ar[d]^{s\circ f} &X\ar[d]^f\\
Y \ar[r]^s & C\ar[r]^\pi &Y\\
}\]
induces \xymatrix{\bL_f \ar[r] \ar@/^1pc/[rr]^{\id}&\bL_{s\circ f} \ar[r] &\bL_f}.
It is then easy to check the following two triangles are isomorphic:
\[
\xymatrix{
f^*\bL_s          \ar[r]\ar[d]    & f^*\bL_s \oplus \bL_f  \ar[r] \ar[d] &        \bL_f\ar[d]\\
f^*\bL_s          \ar[r]            & \bL_{s\circ f} \ar[r]  &        \bL_f.
}
\]
(2)
The isomorphism $L_{s\circ f} \simeq  f^*\bL_s \oplus \bL_f$
induces an isomorphism between intrinsic normal sheaves
\[
N_{s\circ f} \simeq N_f\times_X f^*N_s.
\]

The map $C_{s\circ f} \to C_f \times_X f^*C_s$ is determined by $C_{s\circ f} \to C_f$ and
$C_{s\circ f} \to f^*C_s$, they are induced by 
\[
\xymatrix{
X\ar[r]\ar[d]^{s\circ f} &X\ar[d]^f\\
C\ar[r]^\pi &Y,
}
\]
and
\[
\xymatrix{
X\ar[r]^f\ar[d]^{s\circ f} &Y\ar[d]^s\\
C\ar[r] &C.
}
\]
Thus we have a commutative diagram
\[
\xymatrix{
C_{s\circ f} \ar[r]\ar[d]  & C_f \times_X f^*C_s  \ar[d]\\
N_{s\circ f} \ar[r]          & N_f \times_X f^*N_s.
}
\]
As vertical arrows are closed immersions and the bottom arrow is an isomorphism, the top arrow is a closed immersion.

(3) follows from (2).

\end{proof}

\begin{prop}[Functoriality] \label{prop:func}
Let $f, g$ be DM morphisms, denote their composition by $h$:
\[
\xymatrix{
X \ar[r]^-f \ar@/_/[rr]_-h& Y\ar[r]^-g & S.
}
\] 

Assume we have a compatible triple between POTs
\[ 
\xymatrix{
f^*E_g^{\bullet} \ar[r]\ar[d] & E_h^{\bullet}\ar[d]\ar[r] &E_f^{\bullet} \ar[d]\\
f^*\bL_g          \ar[r]            &  \bL_h \ar[r] &        \bL_f,
}
 \]  i.e., vertical arrows are POTs, and horizontal arrows are distinguished triangles, 
then \[
h^! =f^!\circ g^!.\]

\end{prop}

\begin{proof} 

Denote by $\mathfrak{E}_f$ the vector bundle stack $h^1/h^0({E^{\bullet}_f}^{\vee})$,  similarly we have $\mathfrak{E}_g$, $\mathfrak{E}_h$.

Step 1:
It is enough to show  $h^! =f^!\circ g^! \colon K_0(S) \to K_0(X)$, since the situation is identical under base change.

Consider the map $\varkappa: X \times \bP^1 \to \D_g$ over $\bP^1$ 
and the cartesian diagram
\[
\xymatrix{
X\times \{\infty\} \ar[r]\ar[d]^{s\circ f} 
  & X \times \bP^1\ar[d]^{\varkappa} 
     & X \times \{0\} \ar[l]\ar[d]^{h}\\
C_g \ar[r] \ar[d] 
  & \D_g              \ar[d]
     &  S \times \{0\} \ar[l]\ar[d]\\
\{ \infty \}   \ar[r]^{i_\infty}  
   &  \bP^1                      
      & \{0\} . \ar[l]_{i_0} 
}
\]

In the proof of Theorem 1 in \cite{KKP}, a virtual pullback 
$\varkappa$ is constructed such that
 \[ \varkappa^! = (s \circ f)^!=h^! .\] 
 Here the virtual pullback $(s\circ  f)^!$ is defined by Lemma \ref{split} using the POT of $f$ and $s$, 
 the POT for $s$ corresponds to the closed immersion $C_s \simeq C_g \to \mathfrak{E}_g$. 
Construction of $\varkappa^!$ is recalled in the remark below.

Then argue as in the proof of Theorem 6.5 in \cite{Fu}, we see $h^! =f^!\circ g^!$ follows from
\[
f^!\circ s^! = (s\circ f)^! \colon K_0(C_g) \to K(X),
\] 
or the functoriality
for the map \xymatrix{X \ar[r]^f&Y \ar[r]^s & C_g}. More precisely, for any element $\sF \in K_0(S)$, we can find $\sF^\sim \in K_0(\D_g)$ such that its restriction to $S \times \bA^1$ equals
the pullback of $\sF$ to $S \times \bA^1$, then
\[
i_0^!(\sF^\sim) = \sF, \quad i_{\infty}^!(\sF^\sim) = \sigma_g(\sF).
\]
Here
$i_0^!$ and  $i_{\infty}^!$ are Gysin pullbacks, and $\sigma_g\colon K(S) \to K(C_g)$ is the map defined by deformation
 to the normal cone. 
As  $g^!(\sF) = s^!(\sigma_g(\sF))$, if we assume $f^!\circ s^! = (s\circ f)^!$, then
\[
f^!(g^!(\sF)) = f^! \circ s^!(\sigma_g(\sF))
=(s\circ f)^! \circ i_{\infty}^!(\sF^\sim)=i_\infty^!\circ \varkappa^!(\sF^\sim).
\]
Thus \[ h^!(\sF)=h^! (i_0^!(\sF^\sim))=
i_0!\circ \varkappa^!(\sF^\sim)=i_\infty^!\circ \varkappa^!(\sF^\sim).\]
Here we used $i_t^*\colon K_0(X\times \bP^1) \to K_0(X)$ is independent of $t$.
This follows from $i_t^* =(\pr_X)_*\circ (i_t)_*\circ i_t^*=(\pr_X)_*\circ c_1(\cO_{\bP^1}(1))$.

As we have a cartesian square
\[
\xymatrix{
Y \ar[r]\ar[d]  
   & C_g \ar[d] \\
Y \ar[r]&  \fE_g 
}
\] where the horizontal arrows are zero sections, we only need to prove
functoriality for $X \to Y \to \fE_g$, where the POT for $Y \to \fE_g$
is given by the identification of its intrinsic normal cone with $\fE_g$, and the induced
virtual pullback is the Gysin pullback.

Step 2:
Abusing notation, we use $s\colon X \to \fE_g$ to denote the zero section of 
$\fE_g$.
Consider the cartesian diagram 
\[
\xymatrix{
X \ar[r]^{s_X}\ar[d]^f & f^*\fE_g\ar[d]^-F \ar[r] & X\ar[d]^f\\
Y \ar[r]^{s} &  \fE_g \ar[r] &Y.
}\]
As 
$f^!\circ s^! = s_X^! \circ F^!$ 
by the commutativities of virtual pullbacks,
the identity $f^!\circ s^! = (s\circ f)^!$ is equivalent to  
$s_X^! \circ F^! = (F\circ s_X)^!$, or the functoriality for
\xymatrix{X \ar[r]^-{s_X}& f^*\fE_g\ar[r]^F & \fE_g}. 
Here $s_X^!,  F^!$ are induced from $s, f$ by base change, and we need to check
there is a compatible triple
\[ 
\xymatrix{
E_f^{\bullet} \ar[r]\ar[d] 
   &E_f^{\bullet}\oplus f^*E_s^{\bullet}\ar[d]\ar[r] 
      & f^*E_s^{\bullet} \ar[d]\\
{s_X}^*\bL_F          \ar[r]          
   &  \bL_f\oplus f^*\bL_s\simeq \bL_{s\circ f} \ar[r]
      &        \bL_{s_X},
}
\]
which follows from the commutativity of the diagram 
\[
\xymatrix{
\bL_f \ar[r]\ar[d]
  &  \bL_f\oplus f^*\bL_s\ar[d]\ar[r]
     & f^*\bL_s\ar[d]\\
 {s_X}^*\bL_F   \ar[r]          
   &   \bL_{s\circ f} \ar[r]
       &  \bL_{s_X}.
}
\]

Step 3:
By the arguments in Step 1, we see that 
$s_X^! \circ F^! = (F\circ s_X)^!$ follows from
the functoriality for the map 
\[\xymatrix{ X \ar[r] &f^*\fE_g \ar[r] &\fE_F} \]
here $\fE_F \simeq f^*\fE_g\times_X \fE_f$ by our construction.
Now functoriality means the Gysin pullback along $X \to f^*\fE_g\times_X \fE_f$ is the composition of
Gysin pullbacks along $f^*\fE_g \to f^*\fE_g\times_X \fE_f$ and $X \to f^*\fE_g$, and this is known.

\end{proof}

\begin{remark}
We recall the construction of $\varkappa^!$, which is determined by a closed embedding of 
$N_{X\times \bP^1}\D_YS$ into some vector bundle stack.

Consider the  following map between distinguished triangles over $X \times \bP^1$:
\[ 
\xymatrix{
f^*E_g^{\bullet}\otimes \cO_{\bP^1}(-1) \ar[r]^-{\nu}\ar[d] & f^*E_g^{\bullet}  \oplus E_h^{\bullet}\ar[d]\ar[r] & c(\nu)\ar[d]\\
f^*\bL_g \otimes \cO_{\bP^1}(-1)   \ar[r]^-{\mu}            &      f^*\bL_g  \oplus \bL_h \ar[r] &       c(\mu),
}
 \]
where $c(\mu), c(\nu)$ are the mapping cones of $\mu, \nu$ resp. 
$\mu$ is defined as the composition
\[
\xymatrix{ 
f^*\bL_g \otimes \cO_{\bP^1}(-1)   \ar[r]^-{(T,U)} 
& 
f^*\bL_g  \otimes  (\cO_{\bP^1}\oplus\cO_{\bP^1} ) \simeq f^*\bL_g  \oplus  f^*\bL_g  \ar[r]^-{(\id, \can)}
& 
f^*\bL_g  \oplus \bL_h .
}
\]
Here $T$ and $U$ are homogeneous coordinates on $\bP^1$, $\can$ is the canonical map  $f^*\bL_g \to \bL_h$.
The map $\nu$ is defined similarly.

It is easy to check $c(\nu)$ is a two term complex of vector bundles
as there is a distinguished triangle
\[
f^*E_g^{\bullet} \otimes \cO_{\bP^1}(1) \to c(\nu) \to E_{f}^{\bullet}.
\] Here $\cO_{\bP^1}(1)$ comes from the exact sequence
\[
\xymatrix{
\cO_{\bP^1}(-1)   \ar[r]^-{(T,U)}  & \cO_{\bP^1}\oplus\cO_{\bP^1} \ar[r] &\cO_{\bP^1}(1) }.
\]

Note that $c(\nu) \to c(\mu)$ is 1-connective, or its cone sits in degree $\le -2$, therefore we have
 a closed immersion:
\begin{equation} \label{immersion}
h^1/h^0(c(\mu)^{\vee}) \to h^1/h^0(c(\nu)^{\vee}).
\end{equation} 
Recall \cite[Proposition 1]{KKP} says that 
\[h^1/h^0(c(\mu)^{\vee}) \simeq N_{X\times \bP^1}\D_YS,\] 
so \eqref{immersion} embeds 
$N_{X\times \bP^1}\D_YS$ into a vector bundle stack $h^1/h^0(c(\nu)^{\vee})$.

\end{remark}

\begin{rmk}
When $X \to Y \to S$ are regular closed embeddings, we obtained functoriality
for Gysin pullbacks. 
\end{rmk}


\subsection{Excess intersection formula}
Assume $f$ is a closed imbedding and consider a POT $E^{\bullet} \to \bL_{f}$ for $f$. Since $h^0(E^{\bullet})=h^0(\bL_f)=0$, we can assume $E^{\bullet} = E[1]$, where $E$ is locally free sheave.

\begin{prop}  
\label{prop:excess}
Assume $f$ is a closed imbedding, $E[1] \to \bL_{f}$ a POT, where $E$ is a locally free sheaf.
We have an excess intersection formula,

\[ f^{!}f_{*} =  \Lambda_{-1}(E)
\]
\end{prop}
%
\begin{proof}
Consider the cartesian diagram
\[\xymatrix{
X \ar[r]\ar[d]& X\ar[d]^f\\
X \ar[r]^f & Y.
}\]
Use the fact that virtual pullbacks and push forwards commute.

\end{proof}

\subsection{Remarks}

\subsubsection{}
\label{rmk}

To define a virtual pullback on $f\colon X\to Y$, we have assumed $Y$ is qcqs,
When $Y$ is only quasi-separated\footnote{This is usually built into the definition of algebraic stacks in the literature.}, but 
$f$ is a composition of  a map $\tilde{f}\colon X \to Z$ with $Z$ being of finite type and quasi-separated over $k$, and an \'etale
map  $j\colon \colon Z \to Y$, then we can still define the map
$\sigma_{\tilde{f}}\circ j^!  \colon K_0(Y) \to K_0(C_f)$.
Note that $C_{\tilde{f}} \simeq C_f$ follows from \cite[Proposition 3.14]{BF}.

Using functoriality, it is easy to check that the map
$\sigma_{\tilde{f}}\circ j^!$ is independant of the factorization $f=j\circ \tilde{f}$, thus by abusing notation we denote the resulting pullback by $\sigma_f$.

Then one can define a pullback $f^!$ as before using $\sigma_f$, as $\sigma_f$ is  the composition of an \'etale pullback and a virtual pullback, it is straightforward to extended results in this section to this slightly more general situation.
\subsubsection{}
Twisted virtual structure sheaves correspond to twisted virtual pullbacks of the form
\[
\id_X^{\cP^\bullet} \circ f^!
\] where $\cP^\bullet$ is a perfect complex on $X$, $\id_X^{\cP^\bullet} \in 
\opk(X \xrightarrow{\id_X} X)$
the bivariant class induced by derived tensoring with $\cP^\bullet$.
Properties of twisted virtual pullbacks follow from those of virtual pullbacks.

\section{A Virtual Localization Formula} \label{s3}

The proof of the virtual localization formula in \cite{GP} can
be streamlined using virtual pullbacks, and an optimal form is obtained in \cite{CKL}. 
The arguments in \cite[Section 3]{CKL} can be used to prove the $K$-theoretic virtual localization formula conjectured in \cite[Conjecture 7.2]{rr},
the keypoint is that a modified POT of the fixed substack is compatibile with the POT of the ambient stack, then
the functoriality of virtual pullbacks gives the virtual localization formula.

\begin{rmk}
The localization formula  \cite[Theorem 5.3.1]{CK}
for dg-schemes is also proved by constructing a virtual pullback $\pi_0(i)^!$.
\end{rmk}

\subsection{Notation and Conventions}
We will use $T$ to denote the torus $\bC^*$.

\subsubsection{}
A $T$-stack $X$ is an algebraic stack $X$ with a $T$ action, a $T$-map $f\colon X \to Y$
between $T$-stacks is a map that respects the $T$ action on $X$ and $Y$.
We will denote $X_T$ the quotient stack $[X/T]$,  and $\pi_X$ the quotient map $X \to X_T$.
For a $T$-map $f$,  we have an induced map $f_T:X_T \to Y_T$ between
$X_T$ and $Y_T$.  

\begin{rmk}
There is an equivalence between the 2-category of $T$-stacks
and the 2-category of stacks over $BT$.
\end{rmk}
\subsubsection{}

For a $T$-stack $X$, $\pi_X^*$ induces an equivalence between
the category of coherent sheaves on $X_T$ and 
the category of $T$ equivariant coherent sheaves on $X$ is 
A $T$-equivariant coherent sheaf $F$ on $X$ corresponds to a coherent sheaf $F_T$ on $X_T$ 
such that $F = \pi_X^*(F_T)$.

Denote by $K^T_0(X)$ the $K$ group of equivariant coherent sheaves on $X$, with $\mathbb{Q}$ coefficients . Via $\pi_X^*$, $K^T_0(X)$ is canonically isomorphic to 
$K_0(X_T)$.

It is easy to show $K^T_0(\Spec\bC)\simeq \bQ[t^{\pm 1}]$, and 
$K^T_0(X)$ is a $K^T_0(\Spec\bC)$  module  as $X_T$ is a stack over $BT$.

Recall $\Lambda_{-1}\colon K_0(X) \to K_0(X)$ is given by $[V] \to \sum_i (-1)^i [\Lambda^i V].$
Its equivariant version $\Lambda_{-1}^T \colon K^T(X) \to K^T(X)$ is simply defined as
$\Lambda_{-1}\colon K_0(X_T) \to K_0(X_T)$.

\subsubsection{}
Given a $T$-map $f\colon X \to Y$, a $T$-equivariant POT $\phi\colon E^{\bullet} \to \bL_{f}$ for $f$ can be identified with a POT for $f_T$ given by
$\phi_T\colon E^{\bullet}_T \to \bL_{f_T}^{\bullet}$.
The virtual structure sheaves  
$\cO_\phi \in  K^T_0(X)$ and 
$\cO_{\phi_T}\in K_0(X_T)$ are related by  
$\cO_\phi = \pi_X^* (\cO_{\phi_T})$, and therefore can be identified via
$\pi_X^*\colon K_0(X_T) \to K^T_0(X)$.

\subsubsection{}
Let $X$ be $T$ stack, DM and of finite type over $\bC$, $X^T$ its fixed substack, we will use $i\colon X^T \to X$
to denote the inclusion of $X^T$ as a substack.

Let  $\phi\colon E^{\bullet} \to \bL_{X}^{\bullet}$ be 
a $T$-equivariant POT for $X$.
 We have a decomposition 
 \[
 i^*{E^{\bullet}} =(i^*{E^{\bullet}} )^\fix\oplus  (i^*{E^{\bullet}} )^\mov
 \] of $ i^*{E^{\bullet}}$ into its fixed and moving parts, which come from $T$-eigensheaves
 of $ i^*{E^{\bullet}}$ with zero and nonzero weights respectively.
 
 We have an induced ($T$-equivariant) POT for $X^T$:
\[
\phi^T \colon (i^*{E^{\bullet}} )^\fix \to (i^*\bL_X)^\fix \to \bL_{X^T}.
\]
(See \cite[Proposition 1]{GP} and \cite[Lemma 3.2]{CKL}.)

\subsection{A Virtual Localization Formula}

\begin{thm}

Assume $X$ is a scheme of finite type over $\bC$ with a $T$ action,  
and 
\[i \colon X^T \to X\]
the inclusion of the $T$ fixed loci.
Let $\phi \colon E^{\bullet} \to \bL_{X}^{\bullet}$  be a  $T$-equivariant POT.
Assume ${N^\vir}^\vee= {(i^*E^\bullet)^\mov}$ has a global resolution $N^{-1}\to N^0$
by locally free sheaves on $X^T$.

Under these assumptions, we have
\[
\cO_X^{\vir} = {i_T}_{*}\lp
\frac{ \cO_{X^T}^{\vir} }{ \Lambda_{-1}^T( [{N^{\vir}}^{\vee}]) }
\rp
\]  in $K^T_0(X)\otimes_{\bQ[t,t^{-1}]}\bQ(t)$.

Here
$\cO_X^{\vir}:= \cO_\phi$, $\cO_{X^T}^{\vir}: = \cO_{\phi^T}$,
$[{N^{\vir}}^{\vee}]= [N^0]-[N^{-1}]$ in $K^T_0(X^T)$.

\end{thm}

\begin{proof}
If we modify the POT for $X^T$ to  
\[
\widetilde{\phi^T}\colon 
(i^*E^{\bullet})^\fix\oplus N^{-1}[1] \to (i^*E^{\bullet})^\fix \to  \bL_{X_T},
\]
then we have a compatible triple
between POTs
\[
\xymatrix{
i^* E^{\bullet}\ar[r]\ar[d]  
  &(i^*{E^{\bullet}} )^\fix\oplus N^{-1}[1] \ar[r]\ar[d]
    & N^0[1]\ar[d]\\
i^*\bL_X          \ar[r]&   \bL_{X^T}  \ar[r]&                      \bL_{i}~ ,
}
\]
here the first row is the direct sum of 
$(i^* E^{\bullet})^\fix  \to (i^*{E^{\bullet}} )^\fix \to 0$ and 
$(i^* E^{\bullet})^\mov  \to N^{-1}[1] \to N^0[1] $.

Let $i_T^{!}$ be the virtual pullback induced by 
$N^0[1] \to \bL_i$.

By Proposition \ref{prop:func}, we have
\[ i_T^{!}\cO_\phi= \cO_{\widetilde{\phi^T}}.
\] 
Since \[ {i_T}_* \colon K^T_0(X^T) \to K^T_0(X)\] as a map between
$K^T_0(\Spec\bC)$ module becomes an isomorphism after tensoring with
$\bQ(t)$ by, e.g., \cite[Theorem 3.3 (a)]{EG}, we see that  $\displaystyle \frac{ i_T^{!}}{ \Lambda^T_{-1}(N^0)}$ is an inverse to $ {i^{}_T}_*$ by Proposition \ref{prop:excess}.

By Lemma \ref{add} below, 
\[  \cO_{\widetilde{\phi^T}} =
     \Lambda_{-1}^T(N^{-1}) \cdot \cO_{\phi^T}
\]

Combine the results above, we see that 

\begin{align*}
\cO_\phi 
& =  {i^{}_T}_* \lp \frac{ i_T^{!} \cO_\phi}{ \Lambda_{-1}^T(N^0) }\rp \\
&={i^{}_T}_*\lp \frac{ \Lambda_{-1}^T (N^{-1}) \cdot  \cO_{\phi^T} }{ \Lambda_{-1}^T (N^0)}\rp,
\end{align*}
and 
this is the same as 
\[
\cO_X^\vir = {i_T}_{*}\lp \frac{\cO_{X^T}^\vir }{ \Lambda_{-1}^T( [{N^{\vir}}^{\vee}])} \rp.
\]
\end{proof}

\begin{remark} For a $T$ scheme $X$,
if $L$ is a line bundle over $X^T$ of nonzero weight $k$, then
$\Lambda_{-1}^T(L) =1 -t^kL$ is invertible in $K(X^T)\otimes_{\bQ} \bQ(t)$.
As $1-t^k L =1-t^k - (L-1)t^k $, $1-t^k$ is invertible, and $L -1$ is nilpotent. 
\end{remark}

\begin{rmk}
 Let $U$ be the complement of $X^T$ in $X$.
 To extend the formula to DM stacks,  what we need is 
$K^T_0(U)\otimes_{Q[t^{\pm 1}]} Q(t)=0$. (Or some other ring in place of $Q(t)$.)
This is certainly true with enough hypotheses. For example, if Riemann-Roch holds for
$[U/T]$, giving an isomorphism between 
$K^T_0(U)$ and $A_*(I[U/T])$,  the Chow group of the inertia stack of $[U/T]$, 
then $1-t$ is nilpotent on $K^T_0(U)$, as the Chern character of $1-t$ is nilpotent on
$A_*(I[U/T])$.

\end{rmk}


\begin{lem}[cf.{ \cite[Lemma 1]{GP}}] \label{add}
Given a POT $\psi: F^{\bullet} \to \bL_f$ for a DM morphism $f\colon X \to Y$ and a locally free sheaf $E$ over X, the map 
$\psi' \colon F^{\bullet}\oplus E[1] \to F^{\bullet}\to \bL_f$
induces a POT for $f$,
here $F^{\bullet}\oplus E[1] \to F^{\bullet}$ is projection onto the first factor.
The two virtual structure sheaves are related by
\[\cO_{\psi'} =\Lambda_{-1}(E) \cdot  \cO_{\psi}.\]
\end{lem}

\begin{proof}
Let $C_f$ be the intrinsic normal cone of $X$,
$C(E)= \operatorname{Spec}\operatorname{Sym}E$ the cone associated with $E$,
and $\fF$  the vector bundle stack $h^1/h^0({F^\bullet}^{\vee})$ associate to $F^\bullet$.
Then the closed imbedding $C_f \to \fF\times_X C(E)$ induced by $\psi'$ is the composition of
the closed imbedding $C_f \to \fF$ induced by $\psi$ and the closed embedding
$\fF \to \fF\times_X C(E)$ induced by the zero section $X \to C(E)$.
Consider the cartesian diagram
\[\xymatrix{
X \ar[r]^{0_E}\ar[d]^{0_\fF}           &   C(E)\ar[d]^{0_\fF'}\\
\fF \ar[r]^-{0_E'} & \fF \times_X C(E)
}\]
By definition, we have 
$\cO_{\psi} = 0_\fF^!(\cO_{C_f})$ and 
$\cO_{\psi'} = (0_{\fF'}\circ 0_E)^!(0_E')_*(\cO_{C_f})$.
Here $0_\fF^!$ and $(0_{\fF'}\circ 0_E)^!$ are Gysin pullbacks along the zero section of 
$X \to \fF$ and $0_{\fF'}\circ 0_E \colon X \to \fF \times_X C(E)$ respectively.
As \[(0_{\fF'}\circ 0_E)^! =  (0_E)^! \circ 0_{\fF'}^!,\,\text{and} \, 0_{\fF'}^! \circ (0_E')_*  = (0_E)_* \circ 0_\fF^!, \]
we see that 
\[
\cO_{\psi'} = (0_E)^! \circ 0_{\fF'}^! \circ (0_E')_*(\cO_{C_f}) = (0_E)^! \circ (0_E)_* \circ 0_\fF^! (\cO_{C_f}) =
\Lambda_{-1}(E) \cdot \cO_{\psi}.
\]

\end{proof}

\section{A Degeneration formula in DT theory}\label{s4}
In this section, the base field $k$ is $\bC$, the field of complex numbers.

It is straightforward to adapt the arguments in \cite{LW,MPT} to write down a degeneration formula in DT theory.
The difference between the $K$-theoretic version and the Chow version comes from formal group laws, and this is the content of \cite[Lemma 3]{L}.

\subsection{Setup}
We recall the setup in \cite{LW}.
\subsubsection{Simple degenerations}
Let $\pi\colon X \to C$ be a projective morphism from a smooth variety $X$ to a smooth pointed curve $(C, 0)$ such that fibers outside $0$ are smooth, and the fiber over $0$,
$X_0$, is a pushout  $Y_+\coprod_D Y_-$, where $Y_+, Y_-$ are smooth varieties, and
$D$ is a connected smooth divisor in both $Y_+$ and $Y_-$. We will denote 
$D$ by $D_+$ or $D_-$ when it is viewed as a divisor in $Y_+$ or $Y_-$.

Let $N_+$ be the normal bundle of $D_+$ in $Y_+$, and $\Delta=\bP_D(N_+\oplus \cO)$. Denote the zero and infinity section of $\Delta$ by $D_+$ and $D_-$ respectively,
so that the normal bundle $N_{D+/\Delta}=N_+$.

\begin{rmk}
Let $N_-$ be the normal bundle of $D_-$ in $Y_-$, then $N_+ \otimes N_- \simeq \cO_D$. In order to define $\Delta$ the choice of $N_+$ or $N_-$ doesn't matter.
In fact, it is enough to start with $X_0$ as a pushout $Y_+\coprod_D Y_-$ assuming $N_{D/Y_+}\otimes N_{D/Y_-}\simeq \cO_D$.
\end{rmk}

\subsubsection{Expanded degenerations}
Expanded degenerations are introduced in \cite{Li1}, see \cite[Section 2.5]{GV}  for non-rigid expansions or rubbers. An extensive discussion can be found in 
\cite{ACFW}.

We recall expanded degenerations associated to $X \to C$, relative pairs $(Y_\pm, D_\pm)$, and 
non-rigid expansions of $(D,N_+)$, which will be denoted by
 $\fX\to \fC$,  $(\fY_-, \fD_-) \to \cT$, $(\fD_+, \fY_+) \to \cT$, and 
$(\fD_+, \fY_\sim, \fD_-) \to \cT_\sim$. Expansions of $X_0$ is given by $\fX_0 \to \fC_0$, the fiber of $\fX\to \fC$ over $0\in C$. 
\begin{rmk}
$\cT,\cT_\sim$ are the same as those in \cite{GV,ACFW}. Note that $\fC_0$ is independent of $\fC$, it is the same as $\fT_0$ in \cite{ACFW}.
\end{rmk}

We have the universal family $\fX \to \fC$ of expanded degenerations associated to the family $\pi\colon X \to C$, its singular fibers are expansions of $X_0$ of the form
$X_0[n]$, where
\[
X_0[n]=Y_-\coprod_{D_- =D_+} \Delta_1 \coprod_{D_- =D_+}  \cdots \Delta_n 
\coprod_{D_-=D_+} Y_+
\] and $\Delta_i$ are copies of $\Delta$.
There is a commutative diagram
\[
\xymatrix{
\fX \ar[r]\ar[d] & X\ar[d] \\
\fC \ar[r]         &  C
}
\] that is an isomorphism on smooth fibers and on singular fibers contracts the $\Delta_i$ in $X_0[n]$.

 For  the relative pair $(Y_-,D_-)$ and $(Y_+, D_+)$, the universal families of expanded
 degenerations are denoted 
 $(\fY_-, \fD_-) \to \cT$ 
 and $(\fD_+, \fY_+) \to \cT$ respectively. Recall an expansion of $(Y_-, D_-)$ is of the form
 \[
(Y_-[n], D_-[n])=Y_-\coprod_{D_- =D_+} \Delta_1 \coprod_{D_- =D_+}  \cdots \Delta_n.
\] where $D_-[n]$ is $D_-$ in $\Delta_n$.
We have commutative diagrams
\[
\xymatrix{
(\fY_\pm,\fD_\pm) \ar[r]\ar[d] & (Y_\pm, D_\pm) \\
\cT         &  \ ,
}
\] where $\fD_\pm\simeq \cT\times D_\pm$ over $\cT$, fiberwise $\Delta_i$ are contracted.

\begin{rmk}
Notationwise, $\mathfrak{A}_0$ in \cite{LW} is $\cT\times \cT$, so $(\fY_\pm, \fD_\pm)$ defined here differ from those defined in \cite{LW}
by a factor of $\cT$.
\end{rmk}

We also need the family $(\fD_+, \fY_\sim, \fD_-) \to \cT_\sim$ of nonrigid expanded degeneration 
associated to the pair $(D, N_+)$, 
fibers are of the form
\[
(D_+,\Delta[n]_\sim,  D_-)=\Delta_1 \coprod_{D_- =D_+}  \cdots \coprod_{D_-=D_+}\Delta_n.
\] where $D_+, D_-$ in $\Delta[n]_\sim$ comes from $\Delta_1, \Delta_n$ respectively.
The commutative diagram
\[
\xymatrix{
(\fD_+,\fY_\sim, \fD_-) \ar[r]\ar[d] & D \\
\cT_\sim         &  
}
\] is given fiberwise by projections $\Delta_i \to D$.

\begin{rmk}
The stacks 
$\cT$, $\cT_\sim$, and 
$\fC_0$, the fiber of $\fC \to C$ over $0$,
are algebraic stacks, having quasi-compact, separated diagonals, locally of finite type over $k$.
\end{rmk}

\subsubsection{Moduli spaces of admissible ideal sheaves}
Let $H$ be a $\pi$-ample line bundle on $X$.
We will consider moduli spaces of admissible ideal sheaves
\footnote{
Rank 1 torsion free sheaves with trivial determinants.} with finite automorphism groups on expanded degenerations, denoted in the form $\cM_{\#}^{\#}$,
where superscripts record Hilbert polynomials, and the subscript indicates the family over which the moduli space is considered.

\begin{rmk}
See (\cite[Section 3]{LW}) for discussions on admissibility.
\end{rmk}

For the family $\fX \to \fC$, as it is representable by a projective morphism, we know the Hilbert scheme of this family with Hilbert Polynomial $P$ (with respect to the pullback of $H$ to $\fX$) is an algebraic stack projective over $\fC$, the intersection of its maximal open DM substack and 
its open substack of admissible ideal sheaves is denoted by $\cM^P$.

Similarly, we have $\cM^P_-$, $\cM^{P}_\sim$, $\cM^P_+$, 
all these stacks are proper by {\cite[Theorem 4.14, 4.15]{LW}}.

The fiber of $\cM^P$
over $0 \in C$ is denote by $\cM_0^P$, it is 
the moduli space of admissible ideal sheaves on $\fX_0/\fC_0$ with finite automorphism groups.

From now on we will be only  interested in the case when $\deg P=1$.

Restricting to the divisor $\fD_-$ of $(\fY_-, \fD_-)$ induces an evaluation map
\[
\ev_-\colon \cM^P_- \to \hilb_D.
\]
Here $\hilb_D =\coprod_n \hilb_D^n$ is the Hilbert scheme of points on $D$.

Similarly, we have
\[
\ev_+\colon \cM_+^P\to \hilb_D
\] and
\[
(\ev_+^\sim, \ev_-^\sim) \colon \cM_\sim \to \hilb_D\times\hilb_D .
\]

Given a degree zero polynomial $Q_-$, denote by $\cM^{P,Q_-}_-$ 
the preimage of $\hilb_D^{Q_-}$ under $\ev_-$, where
$\hilb_D^{Q_-}$ is the 
open and closed scheme of $\hilb_D$
parametrizing ideal sheaves with Hilbert polynomial $Q_-$. 
Similarly, we have $\cM_+^{Q_+,P},\cM^{Q_+,P, Q-}_\sim$.

Let $\cM_0 =\coprod_{\deg P=1} \cM_0^P$ be the disjoint union, and similarly
we have $\cM_-,\cM_\sim, \cM_+$.

\subsubsection{Perfect obstruction theories}
Now we need to assume $\pi \colon X \to C$ is a family of 3-folds to ensure 
higher obstruction groups vanish so that we have perfect obstruction
theories.

Consider the family $\fX \to \fC$,  we have the moduli space $\cM\to \fC$ of admissible ideal sheaves. Denote by $\cI \subset \cO_{\cM \times_{\fC} \fX}$ its
universal ideal sheaf.
The dual of the perfect obstruction theory is given by
\[
\bT_{\cM/\fC} \to R{\pi_\cM}_*\rhom(\cI,\cI)_0[1] 
\] which is induced by the Atiyah class of $\cI$, 
where $\bT_{\cM/\fC}=\bL_{\cM/\fC}^{\vee}$ is the tangent complex of $\cM \to \fC$ and $\pi_\cM$
is the projection $\cM \times_{\fC} \fX \to \cM$. (See \cite[Section 4]{HT}, \cite[Propostion 6.1]{LW}.)
 
POTs for $\cM_0,\cM_-,\cM_\sim, \cM_+$ are defined in the same way.
We will use $\cO_{\cM_0}^\vir$, $\cO_{\cM_-}^\vir$, $\cO_{\cM_\sim}^\vir$, $\cO_{\cM_+}^\vir$ 
to denote their corresponding virtual structure sheaves.

\begin{rmk}
The tangent-obstruction complex is given by 
$\rhom(I,I)_0[1]$ at an ideal sheaf $I$.
On the smooth scheme $\hilb_D$, $\rhom(I,I)_0[1]$ is quasi-isomorphic to the tangent space of $I$ in $\hilb_D$ (\cite[p.912]{LW}).
\end{rmk}

\subsubsection{Decomposition of $\cM_0^P$}
There is a natural map
\[
\cT\times \cT \to \fC_0
\]
that pointwise corresponds to gluing 
$Y_-[n]$ and $Y_+[m]$ to 
\[
X_0[n+m] \simeq Y_-[n] \coprod_{D_-[n]=D_+[m]}Y_+[m].
\]
Similarly, we have
maps
\[
\gl_k\colon \cT\times
\underbrace{\cT_\sim \times \cdots \times\cT_\sim}_{k\, \mathrm{factors}}
\times \cT \to \fC_0
\]
that glue expansions of $Y_\pm$ and nonrigid expansions of $(D,N_+)$ to
expansions of $X_0$.

\begin{rmk}
Smooth locally, $\fC_0$ is given by the simple normal crossing divisor $\cup_{i=1}^n D_i$ in $\bA^n =\Spec k[x_1,x_2,...,x_n]$, where $D_i$ is the smooth divisor $(x_i)$.
The map $\gl_k$ is then given by
\[
\coprod_{J,  |J|=k+1} \cap_{j \in J} D_j \to \cup D_i \subset \bA^n.
\]
In particular, the maps $\gl_k, k\ge 0$ are representable and finite.
\end{rmk}

It follows from the definition of fiber product that
the  diagram
\begin{equation}
\xymatrix{
 \cM_- \times_{\hilb_D}  \cM_\sim 
 \times_{\hilb_D} \cdots \cM_\sim \times_{\hilb_D} \cM_+
 \ar[r]^-{\kth}\ar[d]
 &  \cM_0\ar[d]\\
\cT\times \cT_\sim \times \cdots \times\cT_\sim \times \cT \ar[r]^-{\gl_k}
   &\fC_0 
}
\end{equation} is cartesian.
For example, when $k=1$, the fiber product is given by the fiber product
\[
\xymatrix{
&& \cM_+ \ar[d]^-{\ev_+}\\
& \cM_\sim \ar[r]^-{\ev_-^\sim} \ar[d]^-{\ev_+^\sim} &\hilb_D\\
\cM_-\ar[r]^-{\ev_-} & \hilb_D .
}
\]

For ease of notation, we abbreviate $\cT\times \cT_\sim \times \cdots \times\cT_\sim \times \cT$ to $\fC_0[k]$
 and 
 $
 \cM_- \times_{\hilb_D}  \cM_\sim 
 \times_{\hilb_D} \cdots \cM_\sim \times_{\hilb_D} \cM_+
$ to $\cM_0[k]$.

If we use $\cM_0^P$ in place of $\cM_0$, then the fiber product is the disjoint union of
\[
 \cM_-^{P_0, Q_0} 
 \times_{\hilb_D^{Q_0}}  
 \cM_\sim^{ Q_0,P_1, Q_1} 
 \times_{\hilb_D^{Q_1}}
  \cdots 
  \cM_\sim^{Q_{k-1},P_k,Q_k} 
  \times_{\hilb_D^{Q_k}} \cM_+^{Q_k, P_\infty}
\]
over all $(P_0, P_1, \dots P_k, , Q_0, \dots, Q_k,P_\infty)$
such that $\sum_{i=0}^k P_i - \sum_{j=0}^k Q_j+P_\infty=P$.

For a tuple $\delta=(P_0, P_1, \dots P_k, , Q_0, \dots, Q_k,P_\infty)$, 
let 
\[
k(\delta)=k,
\]
 and 
\[P(\delta)=\sum_{i=0}^k P_i - \sum_{j=0}^k Q_j+P_\infty.\]
We denote the module space
\[
 \cM_-^{P_0, Q_0} 
 \times_{\hilb_D^{Q_0}}  
 \cM_\sim^{ Q_0,P_1, Q_1} 
 \times_{\hilb_D^{Q_1}}
  \cdots 
  \cM_\sim^{Q_{k-1},P_k,Q_k} 
  \times_{\hilb_D^{Q_k}} \cM_+^{Q_k, P_\infty}
\] by $\cM_\delta$ and 
 and the gluing map
\[
\cM_\delta \to \cM_0^P.
\] by $\iota_\delta$. 
\begin{rmk}
Given $P$ there are only finitely many $\delta$ such that $\delta(P)=\delta$ and $\cM_\delta$ is nonempty.
\end{rmk}
\subsection{A Degeneration Formula}

Consider the diagram
\[
\xymatrix{
\cM_0[k] \ar[r]\ar[d] 
 & \cM_- \times \underbrace{\cM_\sim \times \cdots \cM_\sim}_{k\ \mathrm{factors}} \times \cM_+  \ar[d]
\\
 \hilb_D^{\times  k} \ar[r]^-{\Delta_{\hilb_D}^{\times  k}} \ \
 & (\hilb_D \times \hilb_D)^{\times  k}.
}
\]
Recall \[
\cM_0[k] =\coprod _{\delta, k(\delta)=k}\cM_\delta,\] 
denote the  component
of 
the Gysin pullback
\[
(\Delta_{\hilb_D}^{\times  k})^! 
(\cO_{\cM_-}^\vir 
 \boxtimes \cO_{\cM_\sim}^\vir
 \boxtimes \cdots
 \cO_{\cM_\sim}^\vir 
 \boxtimes \cO_{\cM_+}^\vir)
\]
on $\cM_\delta$ by $\cO_{\cM_\delta}^\vir$, where

 Consider the diagram
\[
\xymatrix{
\cM_0[k] \ar[r]\ar[d]  &  \cM_0 \ar[d] \\
\fC_0[k]  \ar[r]          & \fC_0 ,
}
\] where $k\ge 0$.
The POT of $\cM_0\to \fC_0$ induces a virtual pullback, 
and 
$\cO_{\cM_0}^\vir$ is obtained by pulling back $\cO_{\fC_0}$.

The degeneration formula is 
\begin{thm}
Let $X \to C$ be a simple degeneration of 3 folds, $P$ a degree $1$ polynomial, 
For any $\delta=(P_0, P_1, \dots P_k, , Q_0, \dots, Q_k,P_\infty)$ satisfiying $\delta(P)=P$ and $\cM_\delta$ nonempty,
let $
\cO_{\cM_\delta}^\vir \in K_0(\cM_\delta)$ be
\[
( \Delta_{\hilb_D}^{\times  k})^! 
(\cO_{\cM_-^{P_0, Q_0}}^\vir 
 \boxtimes \cO_{\cM_\sim^{Q_0,P_1,Q_1}}^\vir
 \boxtimes \cdots
 \cO_{\cM_\sim^{Q_{k-1},P_k.Q_k}}^\vir 
 \boxtimes \cO_{\cM_+^{Q_k,P_\infty}}^\vir),
\]
 then we have
\[
\sum_{k=0}^\infty \sum_{\substack{
\delta\\ P(\delta)=P\\ k(\delta)=k}
} (-1)^k(\iota_\delta)_*\cO_{\cM_\delta}^\vir =\cO_{\cM_0^P}^\vir
\] in $K_0(\cM_0^P)$. Note that by the boundedness of $\cM_0^P$, the left hand side is a finite sum.
\end{thm}

\begin{proof}
It follows from 
the arguments in \cite[Proposition 6.5]{LW}, \cite[3.9]{MPT}
and the functoriality of virtual pullbacks that 
$\cO_{\cM_0[k]}^\vir =\coprod \cO_{\cM_\delta}^\vir$ can be identified with the virtual pullback of  $\cO_{\fC_0[k]}$.

By \cite[Lemma 3]{L}, we have
\[
\sum_{k=0}^\infty (-1)^k(\gl_k)_*\cO_{\fC_0[k]} = \cO_{\fC_0}.
\]
Then by commutativity between virtual pullbacks and proper pushforwards, the theorem is proved.
\end{proof}

\begin{rmk}
We have deformation invariance for the family $\cM_\fC$.
Denote by $i_c$ the regular imbedding of the closed point $c$ to $C$, and form a cartesian diagram as follows:
\[
\xymatrix{
\cM_t^P \ar[r]\ar[d]  &   \fC_c \ar[r]\ar[d] & c \ar[d]^{i_c}  \\
\cM^P  \ar[r]   &  \fC\ar[r]              &  C.
}
\]
Then 
\[
i_c^! \cO_{\cM^P}^\vir =\cO_{\cM_c^P}^\vir.
\]
For $c\ne 0$, $\fC_c$ is a point, and $\cM_c^P$ is the DT moduli space of ideal sheaves on the smooth 3 fold $X_c$ with virtual structure sheaf $\cO_{\cM_c^P}^\vir$.

\end{rmk}

\begin{rmk}
As $\fC_0$ is not quasi-compact, we need to 
use virtual pullbacks explained in subsubsection \ref{rmk}.
This is possible because $\cM_0^P$ is bounded.
\end{rmk}

\subsection*{Acknowlegements}
We would like to thank 
Q. Chen,
A. Kresch,
C. Manolache,
P. Sun,
J. Wise, W. Zheng
for correspondence and/or conversations.


\begin{thebibliography}{99}

\bibitem{AOV}
D. Abramovich, M. Olsson, A. Vistoli,
{\em Twisted stable maps to tame artin stacks},
J.\ Algebraic Geom. {\bf 20} (2011), no.3,  399-477.


\bibitem{ACFW}
D. Abramovich, C. Cadman, B. Fantechi, J. Wise, 
{\it Expanded degenerations and pairs},
Comm. Algebra {\bf 41} (2013), no. 6, 2346-2386. 

 \bibitem{AS}
D.~Anderson, S.~Payne, 
\emph{Operational K-theory},
Doc. Math. {\bf 20} (2015), 357-399.



\bibitem{BF}

K. Behrend, B. Fantechi, 
\emph{The intrinsic normal cone}. 
Invent. Math. {\bf 127}, (1997), no. 1, 45-88.


\bibitem{CKL}
H.-L. Chang, Y.-H. Kiem, J. Li,
{\em Torus localization and wall crossing for cosection localized virtual cycles},
 Adv. Math. {\bf 308} (2017), 964-986.
  

\bibitem{CK}  
I. Ciocan-Fontanine, M. Kapranov, 
{\it Virtual fundamental classes via dg-manifolds}, Geom. Topol. {\bf 13} (2009), no. 3, 1779-1804.
 

  
  \bibitem{CG}
N. Chriss,  V. Ginzburg,
\emph{Representation Theory and Complex Geometry},
  Birkh\"auser, Boston ,1997.

\bibitem{EG}
D. Edidin, W. Graham, 
{\em
Nonabelian localization in equivariant K-theory and Riemann-Roch for quotients
}
Adv. Math. {\bf 198} (2005), no. 2, 547-582. 
 
    \bibitem{rr}
       B.\ Fantechi, L.\ G\"ottsche,
  \emph{Riemann-Roch theorems and elliptic genus for virtually smooth schemes},
  Geom. ~Topol. {\bf 14} (2010) , 83-115.
  
  \bibitem{Fa}
   G. Faltings,
   {\it Finiteness of coherent cohomology for proper fppf stacks},
   J. Algebraic Geom. 12 (2003), no. 2, 357-366.
 
   \bibitem{Fu}
  W.~Fulton,
 \emph{Intersection theory},  
 2nd ed., Springer-Verlag, Berlin, 1998.
 
 \bibitem{Giv}
A.\ Givental,
\emph{On the WDVV equation in quantum $K$-theory},
Michigan Math. J. {\bf48} (2000), 295-304. 

 
 \bibitem{GP}
 T.\ Graber, R.\ Pandharipande,
  \emph{Localization of virtual classes},
  Invent.~Math.~{\bf 135} (1999) 487-136.
 \bibitem{GV}
T.\ Graber, R. Vakil,
{\em Relative virtual localization and vanishing of tautological classes on moduli spaces of curves},
Duke Math. J. \textbf{130} (2005), no. 1, 1-37.

 
 \bibitem{HR0}
 J. Hall, D. Rydh, 
 {\em Algebraic groups and compact generation of their derived 
 categories of representations}, Indiana Univ. Math. J. {\bf 64} (2015), 1925-1958.
 
 \bibitem{HR}
 J. Hall, D. Rydh,
 {\em Perfect complexes on algebraic stacks},
 Compos. Math. {\bf 153} (11) (2017), 2318-2367.
 \bibitem{HT}
 D. Huybrechts, R. Thomas,
{\em Deformation-obstruction theory for complexes via Atiyah and Kodaira-Spencer classes}, Math. Ann. {\bf 346} (2010), no. 3, 545-569.
 
 \bibitem{KKP}
 B. ~Kim, A. ~Kresch and T. ~Pantev, 
 \emph{Functoriality in intersection theory and a conjecture of Cox, Katz and Lee}, J. \ Pure and Appl. 
 \ Algebra {\bf 179} (2003), no. 1-2, 127-146.
 

   \bibitem{K1}
   A.\ Kresch,
\emph{Canonical rational equivalence of intersections of divisors},
 Invent. \ Math. {\bf 136} (1999), no.\ 3, 438-496.
   \bibitem{K2}
   A.\ Kresch,
 \emph{Cycle groups for Artin stacks},
 Invent. \ Math. {\bf 138} (1999), no.\ 3, 495-536.
 
 \bibitem{LMB}
 G. Laumon, L. Moret-Bailly,
{\it Champs alg\'ebriques}, Springer-Verlag, Berlin, 2000.

 \bibitem{L}
Y.~P.~Lee,
\emph{ Quantum K-Theory, I: Foundation},
  Duke Math. J. {\bf 121} (2004), no. \ 3,  389-424.
  

  
  \bibitem{Li1}
J.\ Li, 
\textit{Stable morphisms to singular schemes and relative stable morphisms},
J.\ Differential Geom. \textbf{57} (2001), no.\ 3, 509-578. 




\bibitem{LW}
J. Li, B. Wu,  
 {\em Good degeneration of Quot-schemes and coherent systems},
 Comm. Anal. Geom. {\bf 23} (2015), no. 4, 841-921. 

\bibitem{MNOP}
D. Maulik, N. Nekrasov, A. Okounkov, R. Pandharipande, 
{\em 
Gromov-Witten theory and Donaldson-Thomas theory. II.}
Compos. Math. {\bf 142} (2006), no. 5, 1286-1304.

\bibitem{MPT}
D. Maulik, R. Pandharipande, R. Thomas,
{\em Curves on K3 surfaces and modular forms,
With an appendix by A. Pixton},
J. Topol. {\bf 3} (2010), no. 4, 937-996.
  \bibitem{vpb}
  C.\ Manolache,
  \emph{Virtual pull-backs},
J.\ Algebraic Geom.\ \textbf{21} (2012), no.\ 2, 201-245.
 \bibitem{Ok}
A. Okounkov,
{\it Lectures on $K$-theoretic computations in enumerative geometry}, arXiv:1512.07363v1.
 \bibitem{O1}
 M. Olsson,
  {\it Sheaves on Artin stacks} 
  J. Reine Angew. Math. {\bf 603} (2007), 55-112.
 \bibitem{O2}
  M. Olsson,
  {\it On proper coverings of Artin stacks},
Adv. Math. {\bf 198} (2005), no. 1, 93-106.
  
 \bibitem{Po}
  F. Poma,
  \emph{Virtual classes for Artin stacks},
 Manuscripta Math. {\bf 146} (1) (2015), 107-123.
  \bibitem{Q}
 F. Qu,
{\it Deformation to the normal cone for algebraic stacks}, in preparation.

\bibitem{Ro}
M. Romagny,
\emph{ Group actions on stacks and applications},
Michigan Math. J. {\bf 53} (2005), no.~1, 209-236.
 
\bibitem{Stacks}
The Stacks Project authors,
{\it The Stacks Project}.






\end{thebibliography}
\end{document}